\documentclass[a4paper,10pt,reqno]{amsart}
\usepackage{latexsym, stmaryrd, mathtools}
\usepackage{tikz}
\usetikzlibrary{snakes}
\usetikzlibrary{matrix}

\DeclareMathAlphabet{\mathitbf}{OML}{cmm}{b}{it}

\newcommand{\emptyword}{\ensuremath{\epsilon}}
\newcommand{\B}{\mathcal{B}} 
\newcommand{\barB}{\bar{\B}} 

\theoremstyle{plain}
\newtheorem{theorem}{Theorem}
\newtheorem*{theorem*}{Theorem}

\newtheorem*{corollary*}{Corollary}
\newtheorem{lemma}[theorem]{Lemma}
\newtheorem*{lemma*}{Lemma}

\newtheorem*{proposition*}{Proposition}

\newtheorem*{conjecture*}{Conjecture}

\theoremstyle{definition}

\newtheorem*{definition*}{Definition}

\newtheorem*{example*}{Example}

\newtheorem*{problem*}{Problem}

\theoremstyle{remark}

\newtheorem*{remark*}{Remark}

\newcommand{\lra}{\longrightarrow}

\DeclareMathOperator{\take}{\mathrm{take}}

\DeclareMathOperator{\leaves}{\mathrm{leaves}}
\DeclareMathOperator{\internal}{\mathrm{int}}
\DeclareMathOperator{\sub}{\mathrm{sub}}
\renewcommand{\root}{\operatorname{root}}
\DeclareMathOperator{\rpath}{\mathrm{rpath}}
\DeclareMathOperator{\rsub}{\mathrm{rsub}}

\DeclareMathOperator{\NODES}{\mathtt{v}}
\DeclareMathOperator{\ROOT}{\mathtt{root}}
\DeclareMathOperator{\RPATH}{\mathtt{rpath}}
\DeclareMathOperator{\RSUB}{\mathtt{rsub}}
\DeclareMathOperator{\SUB}{\mathtt{sub}}
\DeclareMathOperator{\LEAVES}{\mathtt{leaves}}
\DeclareMathOperator{\INTERNAL}{\mathtt{int}}

\newcommand{\btree}{\mbox{\ensuremath{\beta(1,0)}-tree}}
\newcommand{\btrees}{\btree s}

\newcommand{\cfilll}{white}

\newcommand{\ns}{4pt}
\newcommand{\nodestyle}{\tikzstyle{every node} = [font=\footnotesize]}
\newcommand{\discstyle}{\tikzstyle{disc} = 
  [ circle,thin,fill=\cfilll,draw=black, minimum size=\ns, inner sep=0pt ] }
\newcommand{\style}{
  \nodestyle
  \discstyle
}

\newcommand{\scl}{0.45}

\newcommand{\leaf}{
  \begin{tikzpicture}[ scale=\scl, baseline=-2.5pt ]
    \style
    \node [disc] (r) at (0,0) {};
  \end{tikzpicture}
}
\newcommand{\edge}{  
  \begin{tikzpicture}[ xscale=0.4, yscale=0.3, baseline=1.9pt ]
    \style
    \node [disc] (r) at (0,1) {};
    \node [disc] (1) at (0,0) {};
    \draw (r)  -- (1);
  \end{tikzpicture}
}
\newcommand{\bx}{
  \begin{tikzpicture}[ scale=\scl, baseline=-2.5pt ]
    \style
    \node [disc] (r) at (0,0) {};
    \draw (r) node[right=1pt]  {\ensuremath{1}};
  \end{tikzpicture}
}
\newcommand{\bbx}{  
  \begin{tikzpicture}[ scale=\scl, baseline=-2pt ]
    \style
    \node [disc] (r) at (0,1) {};
    \node [disc] (1) at (0,0) {};
    \draw (r) node[above=1pt] 
          {\ensuremath{1}} -- (1) node[below=1pt] {\ensuremath{1}};
  \end{tikzpicture}
}
\newcommand{\bbbx}{  
  \begin{tikzpicture}[ scale=\scl, baseline=(r11.base) ]
    \style
    \node [disc] (r)   at (0, 0) {};
    \node [disc] (r1)  at (0,-1) {};
    \node [disc] (r11) at (0,-2) {};
    \draw 
      (r)   node[left] {1} -- 
      (r1)  node[left] {1} -- 
      (r11) node[left] {1};
  \end{tikzpicture}
}
\newcommand{\bbbxx}{  
  \begin{tikzpicture}[ scale=\scl, baseline=(r2.base) ]
    \style
    \node [disc] (r)  at (   0, 0) {};
    \node [disc] (r1) at (-0.6,-1) {};
    \node [disc] (r2) at ( 0.6,-1) {};
    \draw 
      (r) node[above] {2} -- (r1) node[below] {1} -- 
      (r)                 -- (r2) node[below] {1};
  \end{tikzpicture}
}

\newcommand{\bxxx}{
  \style
  \node [disc] (r1)  at ( 0,-1) {};
  \node [disc] (r11) at (-1,-2) {};
  \node [disc] (r12) at ( 0,-2) {};
  \node [disc] (r13) at ( 1,-2) {};
  \draw 
  (r1) -- (r11) node[below=1pt] {1} 
  (r1) -- (r12) node[below=1pt] {1}
  (r1) -- (r13) node[below=1pt] {1};
}
\newcommand{\eeev}{
  \style
  \node [disc] (r)    at ( 0, 3 ) {};
  \node [disc] (r1)   at ( 0, 2 ) {};
  \node [disc] (r11)  at ( 0, 1 ) {};
  \node [disc] (r111) at (-0.8, 0.1 ) {};
  \node [disc] (r112) at ( 0, 0.1 ) {};
  \draw 
  (r11) -- (r111) node[below=1pt] {1}
  (r11) -- (r112) node[below=1pt] {1};
}

\newcommand{\exampleforest}{
    \path
    node [disc] (1)   at (-2, -1) {} 
    node [disc] (11)  at (-3, -2) {}
    node [disc] (12)  at (-1, -2) {}
    node [disc] (112) at (-2, -3) {}
    node [disc] (111) at (-4, -3) {}
    node [disc] (2)   at ( 2, -1) {} 
    node [disc] (21)  at ( 1, -2) {}
    node [disc] (22)  at ( 2, -2) {}
    node [disc] (23)  at ( 3, -2) {};

    \draw 
    (1)  node[above left=-1pt]  {1} -- (11)  node[above left=-1pt]  {2}
    (1)                      -- (12)  node[below=1pt] {1}
    (11)                     -- (111) node[below=1pt] {1}
    (11)                     -- (112) node[below=1pt] {1}
    (2) node[above right=-1pt] {3}  -- (21)  node[below=1pt] {1}
    (2)                      -- (22)  node[below=1pt] {1}
    (2)                      -- (23)  node[below=1pt] {1};
  }
\newcommand{\exampletree}{
  \node [disc] (r) at (0,0) {}; 
  \exampleforest
  \draw
  (r) node[above=1pt] {4} -- (1)
  (r)                     -- (2);
}

\newcommand{\tri}{
  \filldraw[fill=\cfill, draw=\cdraw] 
  (0,0) -- (-2,-2) -- (2,-2) -- cycle;
  \path
  node [disc] (1)    at (0, 0) {} 
  node [disc] (11)   at (1,-1) {};
}

\newcommand{\extree}{
  \begin{center}
    \begin{tikzpicture}[ scale=\scl, baseline=0pt ]
      \discstyle
      \exampletree
    \end{tikzpicture}
  \end{center}
}

\title{An involution on $\beta(1,0)$-trees}

\author[A. Claesson]{Anders Claesson}
\author[S. Kitaev]{Sergey Kitaev}
\author[E. Steingr\'{\i}msson]{Einar Steingr\'{\i}msson}

\begin{document}

\begin{abstract}
  In [Decompositions and statistics for $\beta(1,0)$-trees and
  nonseparable permutations, Advances Appl. Math. 42 (2009) 313--328]
  we introduced an involution, $h$, on {\btrees}. We neglected,
  however, to prove that $h$ indeed is an involution.  In this note
  we provide the missing proof. We also refine an equidistribution
  result given in the same paper.
\end{abstract}

\maketitle

\thispagestyle{empty}

\section{Introduction}

A \emph{\btree}~\cite{JaGi98} is a rooted plane tree labeled with
positive integers such that
\begin{enumerate}
\item Leaves have label $1$.
\item The root has label equal to the sum of its children's labels.     
\item Any other node has label no greater than the sum of its
  children's labels.
\end{enumerate}
Below is an example of such a tree.\\[-3ex]
\extree

In~\cite{CKS} we introduced an involution, $h$, on {\btrees}. We also
gave a result on the equidistribution of certain statistics on
{\btrees}. A proof that $h$ indeed is an involution was, however, not
given; rather, the proof was said to be found in a forthcoming paper
that never materialised. The proof of the equidistribution was in
fact also omitted. In this note we give the two missing proofs. We
also refine the equidistribution result.

\section{The structure of {\btrees}}

We say a {\btree} on two or more nodes is \emph{indecomposable} if its
root has exactly one child and \emph{decomposable} if it has more than
one child. The {\btree} on one node, $\leaf$ = $\bx\!$, is neither
indecomposable nor decomposable. Let $\B_n$ be the set of all
{\btrees} on $n$ nodes, and let $\barB_n$ be the subset of $\B_n$
consisting of the indecomposable trees. Let $\B_n^k$ be the subset of
$\B_n$ consisting of the trees with root label $k$. For instance,
$$ \B_3=\big\{\;\,\bbbx\;,\;\bbbxx\;\big\}\qquad\;
\barB_3=\B_3^1=\big\{\;\bbbx\;\;\,\big\}\qquad\;
\B_3^2=\big\{\;\bbbxx\;\big\}
$$
Decomposable trees can be regarded as sums of indecomposable ones:
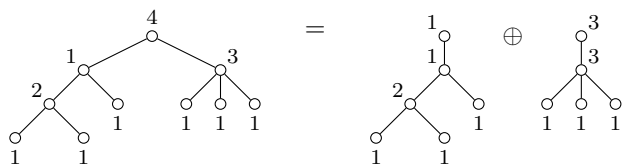
\begin{center}
  \begin{tikzpicture}[ scale=\scl, baseline=0pt ]
    \style    
    \exampletree
  \end{tikzpicture}
  $\quad = \quad$
  \begin{tikzpicture}[ scale=\scl, baseline=0pt ]
    \style
    \path
    node [disc] (r1)  at (-2,  0) {} 
    node [disc] (r2)  at ( 2,  0) {}; 
    \exampleforest
    \draw
    (r1) node[above left=-1pt]  {1} -- (1)   
    (r2) node[above right=-1pt] {3} -- (2);
    \node[font=\normalsize] at (0,0) {$\oplus$};
  \end{tikzpicture}
\end{center}
\smallskip 
In fact we do not need to require $u$ and $v$ to be indecomposable for
the sum $u\oplus v$ to make sense. In general, we define that the root
label of $u\oplus v$ is the sum of the root label of $u$ and the root
label of $v$, and that the subtrees of $u\oplus v$ are those of $u$
followed by those of $v$. So,
$$
\begin{tikzpicture}[ scale=\scl, baseline=(r.base) ]
  \style
  \node [disc] (r) at (0,1) {};
  \node [disc] (1) at (-1,0) {};
  \draw (r) node[above=1pt] 
        {\ensuremath{1}} -- (1) node[below=1pt] {\ensuremath{1}};
\end{tikzpicture}\oplus
\begin{tikzpicture}[ scale=\scl, baseline=(r1.base) ]
  \style
  \node [disc] (r1)  at ( 0,-1) {};
  \node [disc] (r11) at ( 1,-2) {};
  \node [disc] (r12) at ( 0,-2) {};
  \draw 
  (r1) -- (r11) node[below=1pt] {1} 
  (r1) -- (r12) node[below=1pt] {1}; 
  \draw (r1) node[above=1pt] {2};
\end{tikzpicture}
=
\begin{tikzpicture}[ scale=\scl, baseline=(r1.base) ]
  \bxxx
  \draw (r1) node[above=1pt] {3};
\end{tikzpicture} 
=
\begin{tikzpicture}[ scale=\scl, baseline=(r1.base) ]
  \style
  \node [disc] (r1)  at ( 0,-1) {};
  \node [disc] (r11) at (-1,-2) {};
  \node [disc] (r12) at ( 0,-2) {};
  \draw 
  (r1) -- (r11) node[below=1pt] {1} 
  (r1) -- (r12) node[below=1pt] {1}; 
  \draw (r1) node[above=1pt] {2};
\end{tikzpicture} \oplus
\begin{tikzpicture}[ scale=\scl, baseline=(r.base) ]
  \style
  \node [disc] (r) at (0,1) {};
  \node [disc] (1) at (1,0) {};
  \draw (r) node[above=1pt] 
        {\ensuremath{1}} -- (1) node[below=1pt] {\ensuremath{1}};
\end{tikzpicture}
$$ 

Further, there is a simple one-to-one correspondence $\lambda$ between
the Cartesian product $[k]\times\B_{n-1}^k$ and the disjoint union
$\cup_{i=1}^k\barB_{n}^i$, where $\barB_n^k$ is the subset of
$\barB_n$ consisting of the trees with root label $k$:
$$
\begin{tikzpicture}[ scale=\scl, baseline=(r11.base) ]
  \bxxx
  \draw (r1) node[above=1pt] {3};
\end{tikzpicture}
\,\raisebox{2ex}{$\substack{\lambda_1\\ \lra}$}\,
\begin{tikzpicture}[ scale=\scl, baseline=(r11.base) ]
  \bxxx
  \node [disc] (r)  at (0,0) {};
  \draw (r) node[above right=-1pt] {1} -- (r1) node[above right=-1pt] {1};
\end{tikzpicture}
\qquad\;\;
\begin{tikzpicture}[ scale=\scl, baseline=(r11.base) ]
  \bxxx
  \draw (r1) node[above=1pt] {3};
\end{tikzpicture}
\,\raisebox{2ex}{$\substack{\lambda_2\\ \lra}$}\,
\begin{tikzpicture}[ scale=\scl, baseline=(r11.base) ]
  \bxxx
  \node [disc] (r)  at (0,0) {};
  \draw (r) node[above right=-1pt] {2} -- (r1) node[above right=-1pt] {2};
\end{tikzpicture}
\qquad\;\;
\begin{tikzpicture}[ scale=\scl, baseline=(r11.base) ]
  \bxxx
  \draw (r1) node[above=1pt] {3};
\end{tikzpicture}
\,\raisebox{2ex}{$\substack{\lambda_3\\ \lra}$}\,
\begin{tikzpicture}[ scale=\scl, baseline=(r11.base) ]
  \bxxx
  \node [disc] (r)  at (0,0) {};
  \draw (r) node[above right=-1pt] {3} -- (r1) node[above right=-1pt] {3};
\end{tikzpicture}
$$ In general, if $t$ is a tree with root label $k$ and $i$ is an
integer such that $1\leq i\leq k$, then $\lambda_i t$ is obtained
from $t$ by joining a new root via an edge to the old root; and both
the new root and the old root are assigned the label $i$.

Thus each \btree, $t$, is of exactly one the following three forms:
\begin{itemize}
\item[] $t=\leaf$,      \hfill (the single node tree)
\item[] $t=u\oplus v$,   \hfill (decomposable)
\item[] $t=\lambda_i u$, where $1\leq i\leq \root u$,\hfill(indecomposable)
\end{itemize}
in which $u$ and $v$ are \btrees, and $\root u$ denotes the root label of $u$.
As an example of the encoding this
characterisation entails we have
$$    
\begin{tikzpicture}[ scale=\scl, baseline=(111.base) ]
  \style 
  \path
  node [disc] (1)   at ( 0, -1) {} 
  node [disc] (11)  at ( 0, -2) {}
  node [disc] (111) at (-.65, -3) {}
  node [disc] (112) at (0.65, -3) {};
  
  \draw 
  (1) node[right=2pt] {2} -- (11) node[right=2pt] {2}
  (11)                   -- (111) node[below=1pt]     {1}
  (11)                   -- (112) node[below=1pt]     {1};
\end{tikzpicture}
= \lambda_2\Big(
\begin{tikzpicture}[ scale=\scl, baseline=(111.base) ]
  \style 
  \path
  node [disc] (11)  at ( 0, -2) {}
  node [disc] (111) at (-0.65, -3) {}
  node [disc] (112) at ( 0.65, -3) {};
  
  \draw 
  (11) node[above=2pt] {2} -- (111) node[below=1pt] {1}
  (11)                     -- (112) node[below=1pt] {1};
\end{tikzpicture}
\Big)
= \lambda_2\Big(\bbx\oplus\bbx\,\Big)
= \lambda_2\Big(\lambda_1(\leaf)\oplus\lambda_1(\leaf)\,\Big)
$$

\section{An involution on {\btrees}}\label{h}

In this section we define an involution on {\btrees}. To that end we
now describe a new way of decomposing {\btrees}. Schematically the
sum $\oplus$ on {\btrees} is described by
\begin{center}
  \makebox[3.3cm][l]{
  \begin{tikzpicture}[semithick, scale=0.3]
    \def\tri{ -- +(1,-1.732) -- +(-1,-1.732) -- cycle};
    \draw (0,0) \tri;
    \filldraw[fill=white] (0,0) circle (2mm);
    \node[font=\footnotesize] at (0,0.7) {$a$};
    \node at (1.5,-0.6) {$\oplus$};
    \filldraw (3,0) \tri;
    \filldraw[fill=white] (3,0) circle (2mm);
    \node[font=\footnotesize] at (3,0.8) {$b$};
    \node at (5,-0.6) {=};
    \draw[shift={(8,0)}] (0,0) [rotate=-45] \tri;
    \filldraw[shift={(8,0)}] (0,0) [rotate=45 ] \tri;
    \filldraw[fill=white] (8,0) circle (2mm);
    \node[font=\footnotesize] at (8,0.8) {$a+b$};
  \end{tikzpicture}
  }
\end{center}
An alternative sum is\\[-5ex]
\begin{center}
  \makebox[3.3cm][l]{
  \begin{tikzpicture}[semithick, scale=0.3]
    \def\tri{ -- +(1,-1.732) -- +(-1,-1.732) -- cycle};
    \draw (0,0) \tri;
    \filldraw[fill=white] (0,0) circle (2mm);
    \node[font=\footnotesize] at (0,0.7) {$a$};
    \node at (1.5,-0.6) {$\obslash$};
    \filldraw (3,0) \tri;
    \filldraw[fill=white] (3,0) circle (2mm);
    \node[font=\footnotesize] at (3,0.8) {$b$};
    \node at (5,-0.6) {=};
    \draw (7,0) \tri;
    \filldraw (8,-1.732) \tri;
    \filldraw[fill=white] (7,0) circle (2mm);
    \filldraw[fill=white] (8,-1.732) circle (2mm);
    \node[font=\footnotesize] at (7,0.7) {$a$};
    \node[font=\footnotesize] at (8.7,-1.732) {$1$};
  \end{tikzpicture}
  }
\end{center}
That is, to get $u\obslash v$ we join $u$ and $v$ by identifying the
rightmost leaf in $u$ with the root of $v$, and that node is assigned
the label $1$.

The \emph{right path} is the path from the root to the rightmost leaf.
Let $\rpath(t)$ denote the length of (number of edges on) the
right path of $t$. Note that
\begin{align}
  \root(u\oplus v)  &= \root u + \root v\\
  \rpath(u\oplus v) &= \rpath v  \\
  \shortintertext{while}
  \root(u\obslash v)  &= \root u \\
  \rpath(u\obslash v) &= \rpath u  + \rpath v \label{rpath_obslash}.
\end{align}
for $u\neq\leaf$ and $v\neq\leaf$. Thus, with respect to $\obslash$,
$\rpath$ plays the role of $\root$, and vice versa. There is also a
map $\gamma$ that plays a role analogous to that of $\lambda$:
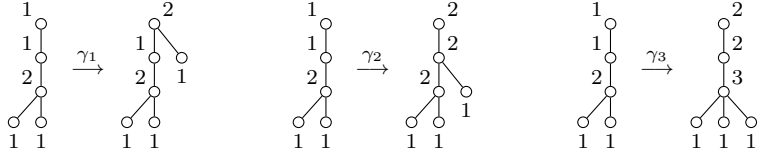
$$
\begin{tikzpicture}[scale=\scl, baseline=(r11.base) ]
  \eeev;
  \draw (r) node[above left=-0.5pt] {1} -- 
  (r1) node[above left=-0.5pt] {1} -- 
  (r11) node[above left=-0.5pt] {2};
\end{tikzpicture}
\;\;\raisebox{2ex}{$\substack{\gamma_1\\ \lra}$}\;
\begin{tikzpicture}[ scale=\scl, baseline=(r11.base) ]
  \eeev;
  \node [disc] (r2) at ( 0.8, 2 ) {};
  \draw (r) node[above right=-0.5pt] {2} -- 
  (r1) node[above left=-0.5pt] {1} -- 
  (r11) node[above left=-0.5pt] {2};
  \draw (r) -- (r2) node[below=1pt] {1};
\end{tikzpicture}
\qquad\quad\;
\begin{tikzpicture}[scale=\scl, baseline=(r11.base) ]
  \eeev;
  \draw (r) node[above left=-0.5pt] {1} -- 
  (r1) node[above left=-0.5pt] {1} -- 
  (r11) node[above left=-0.5pt] {2};
\end{tikzpicture}
\;\;\raisebox{2ex}{$\substack{\gamma_2\\ \lra}$}\;
\begin{tikzpicture}[ scale=\scl, baseline=(r11.base) ]
  \eeev;
  \node [disc] (r12) at ( 0.8, 1 ) {};
  \draw (r) node[above right=-0.5pt] {2} -- 
  (r1) node[above right=-0.5pt] {2} -- 
  (r11) node[above left=-0.5pt] {2};
  \draw (r1) -- (r12) node[below=1pt] {1};
\end{tikzpicture}
\qquad\quad\;
\begin{tikzpicture}[scale=\scl, baseline=(r11.base) ]
  \eeev;
  \draw (r) node[above left=-0.5pt] {1} -- 
  (r1) node[above left=-0.5pt] {1} -- 
  (r11) node[above left=-0.5pt] {2};
\end{tikzpicture}
\;\;\raisebox{2ex}{$\substack{\gamma_3\\ \lra}$}\;
\begin{tikzpicture}[ scale=\scl, baseline=(r11.base) ]
  \eeev;
  \node [disc] (r112) at ( 0.8, 0.1 ) {};
  \draw (r) node[above right=-0.5pt] {2} -- 
  (r1) node[above right=-0.5pt] {2} -- 
  (r11) node[above right=-0.5pt] {3};
  \draw (r11) -- (r112) node[below=1pt] {1};
\end{tikzpicture}
$$ Here is how $\gamma_i t$ is defined in general: Assume that the
length of the right path of $t$ is $k$ and that $i$ is an integer such
that $1\leq i\leq k$. Let us by $x$ refer to the $i$th node on the
right path of $t$. Then $\gamma_i t$ is obtained from $t$ by joining
a new leaf via an edge to $x$, making the new leaf the rightmost leaf
in $\gamma_i t$; and, lastly, add $1$ to the label of each node on
the right path, except for the new leaf. Note that $\rpath\gamma_i t = i$.

We explore the two ways to decompose {\btrees} we now have by defining
an endofunction $h:\B\to\B$ as follows:
\begin{align*}
  h(\leaf)      &= \leaf; \\
  h(\lambda_it) &= \gamma_ih(t); \\
  h(u\oplus v)  &= h(v) \obslash h(u).
\end{align*}
For instance,
$$
\begin{tikzpicture}[ scale=0.35, baseline=-7.5ex ]
  \discstyle
  \exampletree
\end{tikzpicture}\,
\begin{aligned}
\;=\;\;&
  \lambda_1\Big(\lambda_2\big(\,\edge\oplus\edge\,\big)\oplus\edge\,\Big)\oplus
    \lambda_3\big(\,\edge\oplus\edge\oplus\edge\,\big)\\
\raisebox{.6ex}{$\substack{\textstyle{h}\\ \textstyle{\to}}$}\;&
\gamma_3\big(\,\edge\obslash\edge\obslash\edge\,\big)
\obslash\gamma_1\Big(\,\edge\obslash\gamma_2\big(\,\edge\obslash\edge\,\big)\Big)
\;=\,
\end{aligned}
\begin{tikzpicture}[scale=0.35, baseline=(r1.base) ]
  \style
  \node [disc] (r)       at ( 0,   7.4 ) {};
  \node [disc] (r1)      at ( 0,   6.2 ) {};
  \node [disc] (r11)     at ( 0,   5   ) {};
  \node [disc] (r111)    at (-0.8, 4   ) {};
  \node [disc] (r112)    at ( 0.8, 4   ) {};
  \node [disc] (r1121)   at ( 0,   3   ) {};
  \node [disc] (r11211)  at ( 0,   1.8 ) {};
  \node [disc] (r112111) at (-0.8, 0.8 ) {};
  \node [disc] (r112112) at ( 0.8, 0.8 ) {};
  \node [disc] (r1122)   at ( 1.6, 3   ) {};
  \draw 
  (r) node[above=1pt] {2} -- (r1) node[above right=-1pt] {2} 
  (r1) -- (r11) node[above right=-.8pt] {2} 
  (r11) -- (r111) node[left=1pt] {1}
  (r11) -- (r112) node[above right=-.8pt] {1}
  (r112) -- (r1121) node[left=1pt] {1}
  (r1121) -- (r11211) node[above right=-.8pt] {2}
  (r11211) -- (r112111) node[below=1pt] {1}
  (r11211) -- (r112112) node[below=1pt] {1}
  (r112) -- (r1122) node[below=1pt] {1};
\end{tikzpicture}
$$ 
We will soon see that $h$ is in fact an involution!  First we state
some almost self-evident lemmas about relations between $\oplus$,
$\obslash$, $\lambda$, and $\gamma$.

\begin{lemma}\label{LG}
  Let $t$, $u$, and $v$ be {\btrees}. Then
  $$t\oplus(u \obslash v) = (t\oplus u) \obslash v.
  $$
\end{lemma}

\begin{lemma}\label{lG&gL} 
  Let $u$ and $v$ be {\btrees}. Then
  \begin{align*}
    \lambda_i(u \obslash v) &= (\lambda_iu)\obslash v\/;\\
    \gamma_i(u\oplus v) &= u\oplus(\gamma_iv).
  \end{align*}
\end{lemma}

\begin{lemma}\label{lg}
  Let $t$ be a {\btree}. Then
  \begin{align*}
    \gamma_1 t  &= t \oplus \edge\,; \\
    \lambda_1 t &= \edge \obslash t\,;      \\
    \gamma_{i+1}\lambda_{j} &= \lambda_{j+1}\gamma_{i}.
  \end{align*}
\end{lemma}

Next we apply the lemmas above to prove the following lemma which is
the most crucial component in establishing that $h$ is an involution.

\begin{lemma}\label{h-dual}
  Let $t$, $u$, and $v$ be {\btrees}. Then
  $$h(\leaf) = \leaf,\;\;
  h(\gamma_it) = \lambda_ih(t),\,\text{ and }\;
  h(u\obslash v) = h(v)\oplus h(u).
  $$
\end{lemma}

\begin{proof}
  We use induction on the number of nodes. The base case is
  trivial. The proof of the second claim is split into two cases:\\
  Case 1, $t=\lambda_ju$: We shall prove that $h(\gamma_i\lambda_ju) =
  \lambda_i h(\lambda_ju)$ for all positive integers $i$ and $j$. If
  $i=1$, then\smallskip
  \begin{align*}
    h(\gamma_1\lambda_j u)
    &= h(\lambda_j u\oplus \edge)          &&\text{by Lemma \ref{lg}}\\
    &= h(\,\edge\,)\obslash h(\lambda_j u) &&\text{by definition of $h$}\\
    &= \edge \obslash \gamma_jh(u)         &&\text{by definition of $h$}\\
    &= \lambda_1\gamma_j h(u)              &&\text{by Lemma \ref{lg}}\\
    &= \lambda_1h(\lambda_j u)             &&\text{by definition of $h$}
    \shortintertext{If $i>1$, then}
    h(\gamma_i\lambda_ju) 
    &= h(\lambda_{j+1}\gamma_{i-1}u)     &&\text{by Lemma \ref{lg}}\\
    &= \gamma_{j+1}h(\gamma_{i-1}u)      &&\text{by definition of $h$}\\
    &= \gamma_{j+1}\lambda_{i-1}h(u)     &&\text{by induction}\\
    &= \lambda_i\gamma_jh(u)             &&\text{by Lemma \ref{lg}}\\
    &= \lambda_i h(\lambda_ju)           &&\text{by definition of $h$}\\
    \shortintertext{Case 2, $t=u\oplus v$:}
    h\gamma_i(u\oplus v)
    &= h(u\oplus\gamma_iv)               &&\text{by Lemma \ref{lG&gL}}\\
    &= h(\gamma_iv) \obslash h(u)            &&\text{by definition of $h$}\\
    &= \lambda_i h(v) \obslash h(u)          &&\text{by induction}\\
    &= \lambda_i\big(h(v) \obslash h(u)\big) &&\text{by Lemma \ref{lG&gL}}\\
    &= h(u\oplus v)                      &&\text{by definition of $h$}
  \end{align*}  
  The proof of the third claim is also split into two cases.\\
  Case 1, $u=\lambda_it$:\smallskip
  \begin{align*}
    h(\lambda_it \obslash v)
    &= h\lambda_i(t \obslash v)          &&\text{by Lemma \ref{lG&gL}}\\
    &= \gamma_ih(t \obslash v)           &&\text{by Lemma \ref{h-dual}}\\
    &= \gamma_i\big(h(t)\oplus h(v)\big) &&\text{by induction}\\
    &= h(v)\oplus\gamma_ih(t)            &&\text{by Lemma \ref{lG&gL}}\\
    &= h(v)\oplus h(\lambda_it)          &&\text{by definition of $h$}
    \shortintertext{Case 2, $u=s\oplus t$:}
    h((s \oplus t) \obslash v)
    &= h (s\oplus (t \obslash v))              &&\text{by Lemma \ref{LG}}\\
    &= h(t \obslash v) \obslash h(s)           &&\text{by definition of $h$}\\
    &= \big(h(v)\oplus h(t)\big) \obslash h(s) &&\text{by induction}\\
    &= h(v)\oplus\big(h(t) \obslash h(s)\big)  &&\text{by Lemma \ref{LG}}\\
    &= h(v)\oplus h\big(s\oplus t \big)        &&\text{by definition of $h$}
  \end{align*}
\end{proof}

\begin{theorem}
  The function $h$ is an involution.
\end{theorem}

\begin{proof}
  We proceed by induction. By definition $h(\leaf)=\leaf$;
  using that twice the base case follows. For the induction step we
  consider indecomposable and decomposable trees separately.  First, for
  indecomposable trees:
  $$ h^2(\lambda_it)
  = h\big(\gamma_ih(t)\big)
  = \lambda_i h^2(t)=\lambda_i(t).
  $$ Here we have used the definition of $h$, Lemma~\ref{h-dual}, and
  the induction hypothesis. Second, for decomposable trees:
  $$ h^2(u\oplus v) 
  = h\big(h(v)\obslash h(u)\big) 
  = h^2(v)\oplus h^2(u) 
  = v\oplus u.
  $$ Again, we used the definition of $h$, Lemma~\ref{h-dual}, and
  the induction hypothesis.
\end{proof}

\section{Statistics on {\btrees}}\label{tree_stats}

Let $t$ be a \btree. Recall that by $\root t$ we denote the label of
the root. Recall also that the \emph{right path} is the path from the
root to the rightmost leaf, and that the length of the right path is
denoted $\rpath t$.

By $\leaves t$ we denote the number of leaves in $t$; by
$\internal t$ we denote the number of internal nodes (or nonleaves)
in $t$. Note that the root is an internal node.

The number of subtrees (or, equivalently, the number of children of
the root) is denoted $\sub t$.  Further, the number of $1$'s below
the root on the right path is denoted $\rsub t$.

\begin{theorem}\label{thm_h} 
  On {\btrees} with at least one edge, the involution $h$ sends
  the first tuple below to the second.
  $$
  \begin{array}{lllllllll}
    ( &\leaves,    &\internal, &\root,  &\rpath, &\sub,  &\rsub & ) \\
    ( & \internal, &\leaves,   &\rpath, &\root,  &\rsub, &\sub  & )
  \end{array}
  $$
\end{theorem}

\begin{proof}
  We shall use induction to show that $\rpath h(t)=\root t$ and that 
  $\root h(t)=\rpath t$; the other claims follow similarly. The base
  case is trivial. Assume that
  $t=\lambda_iu$ is indecomposable. Then
  $$ \rpath h(\lambda_iu) = \rpath\gamma_ih(u) = i = \root\lambda_iu
  $$ 
  by definition of $h$, definition of $\root$ and $\rpath$, respectively.
  Furthermore, for a decomposable tree $t=u\oplus v$ we have
  \begin{align*}
    \rpath h(u\oplus v) 
    &= \rpath \big(h(u) \obslash h(v)\big) && \text{by definition of $h$} \\
    &= \rpath h(u) + \rpath h(v)    && \text{by \eqref{rpath_obslash}} \\
    &= \root u + \root v            && \text{by induction} \\
    &= \root (u\oplus v)            && \text{by definition of $\root$}
  \end{align*}
  We have thus shown that $\rpath h(t) = \root t$ for any {\btree}
  $t$. That $\root h(t) = \rpath t$ follows from $h$ being an involution.
\end{proof}

The above theorem can be strengthened by introducing what we call
labeled \btrees.
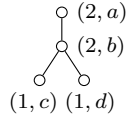
$$
\begin{tikzpicture}[ scale=\scl, baseline=(111.base) ]
  \style 
  \path
  node [disc] (1)   at ( 0, -1) {} 
  node [disc] (11)  at ( 0, -2) {}
  node [disc] (111) at (-.65, -3) {}
  node [disc] (112) at (0.65, -3) {};
  
  \draw 
  (1) node[right=2pt] {$(2,a)$} -- (11) node[right=2pt]  {$(2,b)$}
  (11)                          -- (111) node[below=1pt] {$(1,c)$\;\;}
  (11)                          -- (112) node[below=1pt] {\;\;$(1,d)$};
\end{tikzpicture}
$$
This is a \btree\ in which each node has been assigned a unique
label. In this example, the labels are taken from the alphabet
$\{a,b,c,d\}$.  A recursive characterization of labeled {\btrees} reads
  as follows. A \emph{labeled \btree} is of exactly one of the three
  forms:
\begin{enumerate}
  \setcounter{enumi}{-1}
\item $(1,x)$, a leaf with label $x$;
\item $\lambda((i,x),t)$, where $t$ is a labeled {\btree} and $i\leq\root t$;
\item $u \oplus v$, where $u$ and $v$ are labeled {\btrees}.
\end{enumerate}
Here we assume that the function $\root$ is extended to labeled
{\btrees} by simply ignoring the extra labels. Also, $\lambda$ and
$\oplus$ are extended to labeled {\btrees} in the obvious way:
$$
\begin{tikzpicture}[ scale=\scl, baseline=(111.base) ]
  \style 
  \path
  node [disc] (1)   at ( 0, -1) {} 
  node [disc] (11)  at ( 0, -2) {}
  node [disc] (111) at (-.65, -3) {}
  node [disc] (112) at (0.65, -3) {};
  
  \draw 
  (1) node[right=2pt] {$(2,a)$} -- (11)  node[right=2pt] {$(2,b)$}
  (11)                        -- (111) node[below=1pt] {$(1,c)$\;\;}
  (11)                        -- (112) node[below=1pt] {\;\;$(1,d)$};
\end{tikzpicture}
= \lambda\Big((2,a),\!\!
\begin{tikzpicture}[ scale=\scl, baseline=(111.base) ]
  \style 
  \path
  node [disc] (11)  at ( 0, -2) {}
  node [disc] (111) at (-0.65, -3) {}
  node [disc] (112) at ( 0.65, -3) {};
  
  \draw 
  (11) node[above=2pt] {$(2,b)$}-- (111) node[below=1pt] {$(1,c)$\;\;}
  (11)                             -- (112) node[below=1pt] {\;\;$(1,d)$};
\end{tikzpicture}
\Big)
= \lambda\Big((2,a),\!
  \begin{tikzpicture}[ scale=\scl, baseline=-2pt ]
    \style
    \node [disc] (r) at (0,1) {};
    \node [disc] (1) at (0,0) {};
    \draw (r) node[above=1pt] 
          {\ensuremath{(1,b)}} 
          -- (1) node[below=1pt] {\ensuremath{(1,c)}};
  \end{tikzpicture}
\!\!\oplus\!\!
  \begin{tikzpicture}[ scale=\scl, baseline=-2pt ]
    \style
    \node [disc] (r) at (0,1) {};
    \node [disc] (1) at (0,0) {};
    \draw (r) node[above=1pt] 
          {\ensuremath{(1,b)}} 
          -- (1) node[below=1pt] {\ensuremath{(1,d)}};
  \end{tikzpicture}
\Big)
$$
Similarly, we extend $\gamma$ and $\obslash$:
$$
\begin{tikzpicture}[ scale=\scl, baseline=(111.base) ]
  \style 
  \path
  node [disc] (1)   at ( 0, -1) {} 
  node [disc] (11)  at ( 0, -2) {}
  node [disc] (111) at (-.65, -3) {}
  node [disc] (112) at (0.65, -3) {};
  
  \draw 
  (1) node[right=2pt] {$(2,a)$} -- (11)  node[right=2pt] {$(2,b)$}
  (11)                        -- (111) node[below=1pt] {$(1,c)$\;\;}
  (11)                        -- (112) node[below=1pt] {\;\;$(1,d)$};
\end{tikzpicture}
= \gamma\Big((2,d),
\begin{tikzpicture}[ scale=\scl, baseline=(111.base) ]
  \style 
  \path
  node [disc] (1)   at ( 0, -1) {}
  node [disc] (11)  at ( 0, -2) {}
  node [disc] (111) at ( 0, -3) {};
  
  \draw 
  (1) node[right=2pt] {$(1,a)$} -- (11)  node[right=2pt] {$(1,b)$}
  (11)                        -- (111) node[below=1pt] {$(1,c)$};
\end{tikzpicture}\,
\Big)
= \gamma\Big((2,d),\!
  \begin{tikzpicture}[ scale=\scl, baseline=-2pt ]
    \style
    \node [disc] (r) at (0,1) {};
    \node [disc] (1) at (0,0) {};
    \draw (r) node[above=1pt] 
          {\ensuremath{(1,a)}} 
          -- (1) node[below=1pt] {\ensuremath{(1,b)}};
  \end{tikzpicture}
\!\!\obslash\!\!
  \begin{tikzpicture}[ scale=\scl, baseline=-2pt ]
    \style
    \node [disc] (r) at (0,1) {};
    \node [disc] (1) at (0,0) {};
    \draw (r) node[above=1pt] 
          {\ensuremath{(1,b)}} 
          -- (1) node[below=1pt] {\ensuremath{(1,c)}};
  \end{tikzpicture}
\Big)
$$
The involution $h$ is also easy to extend to \btrees:
\begin{align*}
  h(1,x) &= (1,x); \\
  h\lambda((i,x),t) &= \gamma((i,x),h(t));\\ 
  h(u\oplus v) &= h(v)\obslash h(u). 
\end{align*}
For instance,
$$
\begin{tikzpicture}[ scale=\scl, baseline=(111.base) ]
  \style 
  \node[font=\normalsize] at (-4,-2) {$t\;=$};

  \path
  node [disc] (1)   at (   0, -1) {} 
  node [disc] (11)  at (-0.8, -2) {}
  node [disc] (12)  at ( 0.8, -2) {}
  node [disc] (111) at (-0.8, -3) {}
  node [disc] (121) at ( 0.8, -3) {};
  
  \draw 
  (1) node[above=2pt] {$(2,a)$} -- (11)  node[left =2pt] {$(1,b)$}
  (1)                         -- (12)  node[right=2pt] {$(1,d)$}
  (11)                        -- (111) node[left =2pt] {$(1,c)$}
  (12)                        -- (121) node[right=2pt] {$(1,e)$};

  \draw[->, semithick] (4,-2) -- (5,-2) 
  node[font=\normalsize, midway, yshift=3mm] {$h$};

  \path
  node [disc] (1)   at ( 9.0, -1) {} 
  node [disc] (11)  at ( 8.2, -2) {}
  node [disc] (12)  at ( 9.8, -2) {}
  node [disc] (121) at ( 9.0, -3) {}
  node [disc] (122) at (10.6, -3) {};
  
  \draw 
  (1) node[above=2pt] {$(2,e)$} -- (11)  node[left =2pt] {$(1,d)$}
  (1)                         -- (12)  node[right=2pt] {$(1,c)$}
  (12)                        -- (121) node[left =2pt] {$(1,b)$}
  (12)                        -- (122) node[right=2pt] {$(1,a)$};

  \node[font=\normalsize] at (13.5,-2) {$=\;h(t)$};

\end{tikzpicture}
$$ Let $\NODES(t)$ be the word obtained from preorder traversal of
$t$. Also, denote by $w^r$ the reverse of the word $w$. For instance,
with $t$ as above, we have $\NODES(t)=abcde$ and
$\NODES(t)^r=edcba$.

Let $\LEAVES t$ be the subword of $\NODES(t)$
whose letters are labels of leaves of $t$, and let $\INTERNAL t$ be
the subword of $\NODES(t)^r$ whose letters are labels of internal
nodes of $t$. 

Any labeled \btree\ can be written uniquely as a sum of indecomposable
\btrees. If
$t=\lambda((i_1,x_1),t_1)\oplus\dots\oplus\lambda((i_k,x_k),t_k) $ is
so written, we let $\SUB t
=(\,\NODES(t_1),\dots,\NODES(t_k)\,)$. Similarly, assuming that
$t=\gamma((i_1,x_1),t_1)\obslash\dots\obslash\gamma((i_k,x_k),t_k)$ we
let $\RSUB t = (\,\NODES(t_k)^r,\dots,\NODES(t_1)^r\,)$. 

The definition of the statistic $\ROOT t$ is a bit involved:
$\ROOT t$ is a subword of $\LEAVES t$ of length $k=\root t$. More
precisely, we build this word by starting at the root and greedily and
recursively searching for $k$ leaves in its subtrees starting from the
rightmost subtrees; also, we never search for more leaves in a subtree
than the root label of that subtree. A more precise and formal
definition can be found in the proof of Theorem~\ref{thm_h2}. Let
$\RPATH t$ be the subword of $\NODES(v)^r$ whose letters are labels
of the right path of $t$, excluding the leaf. 

With $t$ and $h(t)$ as in the above picture we have
$$
\begin{array}{lclcl}
\LEAVES t   &=& \INTERNAL h(t) &=& ce       \\  
\INTERNAL t &=& \LEAVES h(t)   &=& dba      \\
\ROOT t     &=& \RPATH h(t)    &=& ce       \\ 
\RPATH t    &=& \ROOT h(t)     &=& da       \\
\SUB t      &=& \RSUB h(t)     &=& (bc,de)  \\ 
\RSUB t     &=& \SUB h(t)      &=& (d,cba). \\  
\end{array}
$$

\begin{theorem}\label{thm_h2}
  On labeled {\btrees} with at least one edge, the involution $h$ sends
  the first tuple below to the second.
  $$
  \begin{array}{lllllllll}
    ( &\LEAVES,    &\INTERNAL, &\ROOT,  &\RPATH, &\SUB,  &\RSUB & ) \\
    ( & \INTERNAL, &\LEAVES,   &\RPATH, &\ROOT,  &\RSUB, &\SUB  & )
  \end{array}
  $$
\end{theorem}

\begin{proof} 
  In terms of the recursive decomposition of labeled \btrees, we have
  $$
  \begin{array}{lcl}
  \LEAVES\ (1,x)            &=& x \\
  \LEAVES \lambda((i,x),t)  &=& \LEAVES t \\
  \LEAVES (u\oplus v)       &=& \LEAVES u \,\sqcup\, \LEAVES v \\[1.3ex]
  \INTERNAL\ (1,x)          &=& x \\
  \INTERNAL \gamma((i,x),t) &=& \INTERNAL t \\
  \INTERNAL (u \obslash v)  &=& \INTERNAL v \,\sqcup\, \INTERNAL u\\[1.3ex]
  \ROOT\ (1,x)              &=& x \\
  \ROOT \lambda((i,x),t)    &=& \take_i(\ROOT t) \\ 
  \ROOT (u\oplus v)         &=& \ROOT u \,\sqcup\, \ROOT v \\[1.3ex]
  \RPATH\ (1,x)             &=& x \\
  \RPATH \gamma((i,x),t)    &=& \take_i(\RPATH t) \\
  \RPATH (u \obslash v)     &=& \RPATH v \,\sqcup\, \RPATH u \\[1.3ex]
  \SUB\ (1,x)               &=& \emptyword \\
  \SUB \lambda((i,x),t)     &=& (\NODES(t))   \\
  \SUB (u \oplus v)         &=& \SUB u \,\sqcup\, \SUB v \\[1.3ex]
  \RSUB\ (1,x)              &=& \emptyword \\
  \RSUB \gamma((i,x),t)     &=& (\NODES(t)^r)   \\
  \RSUB (u \obslash v)      &=& \RSUB v \,\sqcup\, \RSUB u 
  \end{array}
  $$ 
  where $u\/\sqcup\/v$ denotes the concatenation of $u$ and $v$, and
  $\take_i(a_1\dots a_n) = a_1\dots a_i$.  Using induction and
  the definition of $h$ the result readily follows.
\end{proof}

\bibliographystyle{plain}

\end{document}